\documentclass[a4paper,12pt]{amsart}
\usepackage{paralist} 
\usepackage[T1]{fontenc}
\usepackage{currvita}
\usepackage{geometry}
\usepackage{amsmath}
\allowdisplaybreaks
\usepackage{amssymb}
\usepackage{amsfonts}
\usepackage{color}
\usepackage{enumerate}

\usepackage{booktabs}

\usepackage{etoolbox}
\patchcmd{\thebibliography}{\section*{\refname}}{}{}{}

\usepackage{verbatim}

\usepackage{graphicx}
\usepackage{subfig}
\usepackage{bm}
\usepackage{tikz}
\bgroup

\numberwithin{equation}{section}

\newtheorem{lemma}{Lemma}[section]
\newtheorem{theorem}[lemma]{Theorem}
\newtheorem{remark}[lemma]{Remark}

\newtheorem{proposition}[lemma]{Proposition}

\newcommand{\cH}{\mathcal{H}}

\newcommand{\R}{\mathbb{R}}

\newcommand{\Z}{\mathbb{Z}}

\newcommand{\cG}{\mathcal{G}}

\newcommand{\W}{\mathcal{W}}

\newcommand{\cS}{\mathcal{S}}

\def\rd{\R^d}
\def\rdd{{\R^{2d}}}

\def\lrd{L^2(\rd)}

\newcommand{\modsp}{modulation space}
\newcommand{\tfa}{time-frequency analysis}
\newcommand{\stft}{short-time Fourier transform}
\newcommand{\tf}{time-frequency}
\newcommand{\tfs}{time-frequency shift}
\newcommand{\psdo}{pseudodifferential operator}

\newcommand{\ud}{\,\mathrm{d}}
\newcommand{\norm}[2]{\left\| #2 \right\|_{#1}}

\definecolor{darkviolet}{rgb}{0.58,0,0.83}
\usetikzlibrary{patterns}

\subjclass[2020]{47G30,42C40,47A10,47L80,35S05}
\date{}

\begin{document}

\title[H\"older-Continuity of Spectra]{H\"older-Continuity of Extreme Spectral Values of Pseudodifferential Operators, Gabor Frame Bounds, and Saturation}

\author[K. Gr\"ochenig]{Karlheinz Gr\"ochenig}
\address[K. G.]{Faculty of Mathematics, 
University of Vienna, 
Oskar-Morgenstern-Platz 1, 
A-1090 Vienna, Austria}

\author[J. L. Romero]{Jos\'e Luis Romero}
\address[J. L. R.]{Faculty of Mathematics,
	University of Vienna,
	Oskar-Morgenstern-Platz 1,	1090 Vienna, Austria, and
	Acoustics Research Institute, Austrian Academy of Sciences, Dominikanerbastei 16, 1010 Vienna,  Austria}
\email{jose.luis.romero@univie.ac.at}

\author[M. Speckbacher]{Michael Speckbacher}
\address[M. S.]{Acoustics Research Institute, Austrian Academy of Sciences,
        Dominikanerbastei 16, 1010 Vienna,  Austria}
\email{michael.speckbacher@oeaw.ac.at}

\dedicatory{Dedicated to Akram Aldroubi on the occasion of his 65th birthday}

\thanks{This research was funded by the Austrian Science Fund (FWF) 10.55776/Y1199. 
For open access purposes, the authors have applied a CC BY public copyright license to any author-accepted manuscript version arising from this submission.}

\begin{abstract}
We build on our recent results on the Lipschitz dependence
  of the extreme spectral  values of one-parameter families of
  pseudodifferential operators with symbols in a weighted Sj\"ostrand class. We prove
  that   larger symbol classes  lead to H\"older continuity with
  respect to the parameter. This result is then used to investigate
  the behavior of frame bounds of families of Gabor systems
  $\mathcal{G}(g,\alpha\Lambda)$ with respect to the parameter
  $\alpha>0$, where $\Lambda$ is a set of non-uniform, relatively
  separated time-frequency shifts,  and $g\in M^1_s(\R^d)$, $0\leq
  s\leq 2$. In particular, we show that the frame bounds depend continuously on $\alpha$ if $g\in M^1(\R^d)$, and are
  H\"older continuous if $g\in M^1_s(\R^d)$, $0<s\leq 2$, with the
  H\"older exponent explicitly given.
\end{abstract} 

\maketitle

\section{Introduction}

The \emph{extreme spectral values} of a bounded self-adjoint operator $T$ are given by $\sigma_+(T)=\max\{\lambda\in\R:\ \lambda\in\sigma(T)\}$ and $\sigma_-(T)=\min\{\lambda\in\R:\ \lambda\in\sigma(T)\}$, where $\sigma(T)$ denotes the spectrum of $T$.
In \cite{grorosp23}, we studied certain one-parameter families of
pseudodifferential operators and the Lipschitz continuity of their
extreme spectral values. While small perturbations in operator norm result in small changes in spectra, the one-parameter families treated in \cite{grorosp23}
include non-norm-continuous operations, such as
dilation of Weyl symbols, and are therefore a veritable deformation theory of \psdo s.

Our work in \cite{grorosp23} was motivated by Bellissard's seminal results \cite{bel94} on the 
almost-Mathieu operator in a non-commutative torus, and extended them to families of operators with possibly non-periodic Weyl symbols by systematically exploiting \tf\ methods and the theory of modulations spaces. See also \cite{copu24} for closely related results.

While many other contributions in the wake of Bellissard focus on spectral problems in mathematical physics, in particular on magnetic \psdo s ~\cite{atmapu10, bebe16, bebeni18, co10, copu12, copu15, cohepu21, ko03, beta21},
our motivation was to understand the behavior of the frame bounds of Gabor frames under a dilation of the underlying lattice.

In this paper, we extend our results and investigate the relation between the smoothness of a symbol class and the H\"older continuity of the extreme spectral values of a \psdo\ in more detail.

To formulate our main results, let us recall the main objects. 
Let $z=(x,\omega)\in\R^{2d}$. The \emph{short-time Fourier transform}
of $f$ with respect to a  window $g$ is given by 
$$
V_gf(z)=\langle f,\pi(z)g\rangle=\int_{\R^d}f(t)\overline{g(t-x)}e^{-2\pi i \omega\cdot t}\ud t,
$$
where $\pi (z)$ is the \tfs \ $\pi(z)g(t)= e^{2\pi i\omega\cdot t}
g(t-x)$.

We let $\varphi(t)=2^{d/4}e^{-\pi |t|^2}$ denote the standard Gaussian in $\R^d$ and use it as the canonical window for the \stft .
 We now define the \emph{mixed-norm weighted modulation space}
 $M^{p,q}_{s,t}(\R^d)$,  $1\leq p,q\leq\infty$, $s,t\in\R,$ as the space of 
 all tempered distributions $f$  in $\mathcal{S}'(\R^d)$ for which  the norm 
\begin{equation}\label{eq:def-mod-sp}
\| f\|_{M^{p,q}_{s,t}}:=\left(\int_{\R^{d}}\left(\int_{\R^{d}}|V_\varphi f(x,\omega)| ^p(1+|x |)^{sp}  \ud x\right)^{q/p} (1+|\omega |)^{tq}\ud\omega \right)^{1/q}
\end{equation}
 is finite,  with the usual modification when $p=\infty$ or $ q=\infty$. If $p = q$, we write
$M^p_{
s,t}(\R^d)$, and if $s = t$, we write $M^{p,q}_s (\R^d)$. \footnote{Warning: while in our notation $M^p_s$ means $M^{p,p}_{s,s}$ it is also common to let $M^p_s$ denote the modulation space defined with respect to the radial weight $(1+|x|+|\omega|)^{s}$.}
Using any nonzero function $g\in \mathcal{S}(\R^d)$ instead of the Gaussian in \eqref{eq:def-mod-sp} gives an equivalent norm, i.e. $\| f\|_{M^{p,q}_{s,t}}\asymp \|V_gf\|_{L^{p,q}_{s,t}}$, with implied constants depending on $s$ and $t$. 

\vspace{3mm}

\noindent \textbf{Pseudodifferential operators.} The Weyl transform
of  a symbol $\sigma\in \mathcal{S}'(\R^{2d})$ is the operator 
\begin{equation}
\sigma^w
f(y)=\int_{\R^{2d}}\sigma\left(\frac{x+y}{2},\omega\right)e^{2\pi i
  (y-x)\cdot \omega}f(x)\ud x\ud\omega  
\end{equation}
for $f\in \cS (\rd )$,  and  a suitable interpretation of the integral. 
Let $D_a $ denote the dilation  $D_a \sigma (z) = \sigma (az)$. 
We study spectral properties of    one-parameter families of \psdo s
with associated symbols 
\begin{equation}
  \label{eq:c1}
\sigma_\delta=D_{\sqrt{1+\delta}}G_\delta,\quad \delta\in(-\delta_0,\delta_0),  
\end{equation}
and  write $T_\delta:=\sigma_\delta^w$ for the corresponding
operators. For such symbols we regard the mapping $\delta \mapsto
G_\delta $ as a small and norm-continuous perturbation of the operator,
whereas the dilation $D_a$  amounts to a large
deformation of the operator that is usually not norm-continuous.

The objective of this article is to  establish conditions that ensure
that the  extreme spectral values of $T_\delta$ are H\"older
continuous with respect to the parameter $\delta$. Our main result is the following extension of \cite[Theorem~1.1]{grorosp23} which established the statement in the special case when $s=2$.
\begin{theorem}\label{lip}
Let $0<\delta_0<1$, and $0< s\leq 2$. For $\delta \in (-\delta_0,\delta_0)$, let
$G_\delta \in M^{\infty,1}_{0,s}(\R^{2d})$ be real-valued and $\delta\mapsto G_\delta  $ be weakly differentiable\footnote{This means 
 there exists $\partial_\delta G_\delta\in\mathcal{S}'(\R^{2d})$ such that $ \frac{d}{d\delta}\langle G_\delta, F\rangle= \langle \partial_\delta G_\delta, F\rangle$ for every $F \in\mathcal{S}(\R^{2d})$. If $s=2$, then this formulation can be interpreted pointwise, i.e. that the  partial derivative of $(z,\delta)\mapsto G_\delta(z)$ with respect to $\delta$ exists for almost every $z\in\R^{2d}$.} such  that  $\partial_\delta
{G_\delta}\in M^{\infty,1}_{0,s-2}(\R^{2d})$.  Let $T_\delta = \sigma _\delta ^w $
be the family of pseudodifferential operators with Weyl symbols $\sigma _\delta =
D_{\sqrt{1+\delta }} G_\delta $.   
Then, for
$\delta_1,\delta_2 \in (-\delta _0, \delta _0)$,
\begin{align*}
|\sigma_\pm(T_{\delta_1} )&-\sigma_\pm(T_{\delta_2})|\\ &\leq C_{d,s} \cdot |\delta_1-\delta_2|^{s/2}\cdot (1-\delta_0)^{-(d+1)} \cdot \sup_{|t|<\delta_0} \big( \left\|  G_t\right\|_{M^{\infty,1}_{0,s}} +\left\| \partial_tG_t\right\|_{M^{\infty,1}_{0,s-2}}  \big),
\end{align*}
where $C_{d,s}$ is a constant that only depends on $d$ and $s$.
\end{theorem}

The symbol class $M^{\infty ,1}_{0,s}$ is a weighted \modsp , called a
Sj\"ostrand class in the theory of \psdo s. We may think of $s$ as a
smoothness parameter roughly indicating that $G\in M^{\infty
  ,1}_{0,s}$ is bounded and $s$-H\"older continuous. Theorem~\ref{lip}
says that the extreme spectral bounds of the associated family of
\psdo s is H\"older continuous with exponent $s/2$.  
For $s=2$ and Lipschitz continuity of the extreme spectral values,
Theorem~\ref{lip} was the main result of~\cite{grorosp23}.

One could now move
through that proof line-by-line and make the necessary
modifications. This procedure would be rather long-winded and
technical. We therefore take a different route  and derive
Theorem~\ref{lip}  as a consequence of  the already established Lipschitz continuity
and  a new  approximation argument. The idea is to  approximate the symbol
$G_\delta $ by a smoother symbol $H_\delta $ in the \modsp\ $M^{\infty
  ,1}_{0,2}$, for which we have Lipschitz continuity. We then  quantify the
approximation error of $G_\delta- H_\delta $ in the $M^{\infty,1}$-norm while
simultaneously controlling the $M^{\infty,1}_{0,2}$-norm of the
approximating symbol $H_\delta $.  

\vspace{3mm}

\noindent \textbf{Gabor frames.} Our main motivation to   consider the specific setup of Theorem~\ref{lip} is its application in Gabor analysis. Let $\Lambda \subseteq \rdd $ be a
discrete set (not necessarily a lattice), and consider the set  of
time-frequency shifts $\cG (g,\Lambda ) = \{\pi (\lambda )g\ :  \lambda
\in 
\Lambda \}$ for some window function $g\in \lrd $. A major problem in \emph{Gabor analysis} is to study  when $\cG (g,
\Lambda )$ is a frame, i.e., when there exist constants $A,B>0$ such that
\begin{equation}
  \label{eq:frame}
A\|f\|_2^2 \leq \sum _{\lambda \in \Lambda } |\langle f, \pi (\lambda
)g \rangle |^2 \leq B \|f\|_2^2, \, \qquad  f\in \lrd .
\end{equation}
 In this article, we  consider the case that the set of time-frequency
 shifts is dilated by a parameter  $\alpha>0$ and study the dependence
 of the frame bounds on the parameter.  It is easy to see that 
the optimal constants  $A(\Lambda )$ and $B(\Lambda )$ in
\eqref{eq:frame} 
are  the   smallest and largest spectral value of the frame operator
$S_{g,\Lambda } f = \sum_{\lambda\in \Lambda } \langle f, \pi
(\lambda )g \rangle \pi (\lambda )g$, i.e.,
$$
A(\Lambda)=\sigma_-(S_{g,\Lambda}) \quad \text{and} \quad
B(\Lambda)=\sigma_+(S_{g,\Lambda}) \, .
$$
In the problem of the deformation of Gabor frames one investigates how
the frame bounds of the deformed Gabor family $\cG (g, \alpha \Lambda
)$ depend on the dilation parameter $\alpha $.  Since the frame
operator $S_{g, \alpha \Lambda }$ can be written as a \psdo\ in the
form \eqref{eq:c1}, the behavior of the frame bounds with respect to
$\alpha$ amounts  to  a problem of (H\"older-)continuity of the
extreme spectral  values
$A(\alpha\Lambda)=\sigma_-(S_{g,\alpha\Lambda})$ and
$B(\alpha\Lambda)=\sigma_+(S_{g,\alpha\Lambda})$.  This allows us to  
 apply Theorem~\ref{lip} and derive  conditions on $g$ that
guarantee  certain modes of continuity.

The study of the influence of dilations of time-frequency sets on the
respective frame bounds has a long history. If $\Lambda $ is   a lattice and $g\in M^1(\R^d)$, it was  shown in 
\cite{FK04} that the frame bounds depend \emph{lower semi-continuously}  on
$\alpha $, which implies that the set of lattices that generate a
Gabor frame is an open set.  For general
non-uniform sets $\Lambda$, the lower semi-continuity of the frame
bounds was proven later in \cite{asfeka14}.

As a first result, we  strengthen the results of \cite{FK04,asfeka14}
and show  that for $g\in M^1(\R^d)$ the frame bounds  actually
depend continuously on $\alpha$. 
\begin{theorem}\label{thm:cont}
Let $g\in M^1(\R^d)$ and $\Lambda\subset\R^{2d}$ be relatively separated.
Then the frame bounds of $\mathcal{G}(g,\alpha\Lambda)$ are continuous at $\alpha=1$, that is
\begin{align*}
|\sigma_\pm(S_{g,\Lambda})-\sigma_\pm(S_{g,\alpha\Lambda})|\to 0, \quad \alpha\to 1.
\end{align*}
\end{theorem}
When $\Lambda$ is a lattice, the frame operator $S_{g,\Lambda}$ has a periodic Weyl symbol, and Theorem \ref{thm:cont} follows from a classical result of Elliott \cite{ell82}; see also \cite[Theorem 1]{bel94}. To cover general time-frequency nodes $\Lambda$ we shall resort to \cite{grorosp23}.

The density theorem for Gabor frames states that every frame $\cG
(g,\Lambda )$ must satisfy the necessary density condition
$D^-(\Lambda ) \geq 1$, where $D^-(\Lambda )$ is the lower Beurling
density of $\Lambda $. If 
$g\in M^1(\R^d)$, then the  strict inequality    $D^-(\Lambda ) >1$
is necessary to obtain a Gabor frame,
see~\cite{asfeka14,grorcero15,heil07}. In particular, if $g\in
M^1(\R^d)$ and $D^-(\Lambda ) = 1$, then  $A(\Lambda ) = 0$. This is
the so-called Balian-Low phenomenon~\cite{BHW95}. For
practical purposes, however, it is  important to quantify how the
lower frame bound   $A(\alpha\Lambda)$ degenerates  to zero as
$\alpha\to 1$ since the ratio 
$B(\Lambda)/A(\Lambda)$ governs the quality of reconstruction methods.  

This question can be fully answered with the help of
Theorem~\ref{lip}  and an approximation argument. We  show that
windows from the  intermediate spaces $M^1_2(\R^d)\subset
M^1_s(\R^d)\subset M^1(\R^d)$, $0<s<2$, yield   
H\"older-continuous dependence of the frame bounds.

\begin{theorem}\label{thm:holder}
Let $g\in M^1_s(\R^d)$, $0<s\leq 2$, and $\Lambda\subset\R^{2d}$ be
relatively separated. Define the function
$\gamma:(0,2]\to(0,1]$ by
\begin{equation}
  \label{eq:c2}
  \gamma (s) =
  \begin{cases}
    \frac{s}{2(4-3s)} &  0<s<1 \\
\frac{s}{2} & 1\leq s \leq 2 \\
1 & s\geq 2
\end{cases}.
\end{equation}
Then the frame bounds of $\mathcal{G}(g,\alpha\Lambda)$ are $\gamma(s)$-H\"older continuous at $\alpha=1$. More precisely,
$$
|\sigma_\pm(S_{g,\Lambda})-\sigma_\pm(S_{g,\alpha\Lambda})|\leq C_{d,s}\cdot \emph{rel}(\Lambda) \cdot|1-\alpha|^{\gamma(s)}\cdot\|g\|_{M^1_s}^2,
\qquad 3/4 < \alpha < 2.
$$
\end{theorem}

Theorem~\ref{thm:holder}  describes the quantitative behavior of the lower frame bound near the critical density $D^-(\Lambda)=1$ for windows in $M_s^1(\R^d)$, $s>0$,
$$
A(\alpha\Lambda)\lesssim \left\{\begin{array}{ll}(1-\alpha)^{s/(2(4-3s))}& 0<s<1,\\ (1-\alpha)^{s/2},& 1\leq s<2,
\\
1-\alpha, & s\geq 2,
\end{array}
\right.
$$
for $\alpha<1$, $\alpha\to 1$.
The nonlinear behavior of the exponent for $0<s<1$ is an artefact of our proof. In fact, with some more effort, one can show that the exponent $s/2$ holds for a larger range of $s$, though we decided not to include the required technicalities in this article.

\vspace{3mm}

\textbf{Saturation.} One may wonder what happens when the smoothness
parameter $s$ in Theorems~\ref{lip} and \ref{thm:holder} becomes very large. Does
the smoothness of the extreme spectral values increase? The answer is
no and follows from some examples in the theory of Gabor frames.

Specifically, for  
  the Gaussian window is was shown in  \cite{BGL10}  that the lower
  frame bound behaves as 
$$
A(\alpha\Z^2)\asymp (1-\alpha), \quad 1/2 \leq \alpha\leq 1,\quad\text{ and }\alpha\to 1.
$$
Similar results were obtained for the exponential functions
$e^{-t}\chi _{[0,\infty )}$ and $e^{- |t|}$ \cite{KS14}.

Since the Gaussian $\varphi (t) = e^{-\pi t^2}$ is a Schwartz
function and therefore in all $M^1_s, s\geq 0$,  Lipschitz continuity
is the strongest possible mode of continuity one can prove in general,
no matter how big the exponent of the weight $s$ is chosen.
In addition, for the Gaussian window, $A(\alpha \Z ^2)$ is not differentiable at $\alpha
=1$, because $A(\alpha\Z^2)\asymp (1-\alpha)$ for $\alpha \leq 1$ and
$A(\alpha\Z^2)= 0$ for $\alpha \geq 1$.

We conclude that the smoothness of the extreme spectral values is
saturated at Lipschitz continuity. While such saturation phenomena are
well known in approximation theory, we are unaware of similar results
in spectral theory. 

\section{Background and Tools}\label{sec:tools}

\subsection{Notation}
Euclidean balls of radius $R$ and center $x$ in $\R^d$ are denoted $B_{R,d}(x)$. We let $\varphi(t)=2^{d/4}e^{-\pi |t|^2}$ denote the standard Gaussian
 in $\R^d$. The dilation operator acts on a
function $f: \mathbb{R}^d \to \mathbb{C}$ by $D_a f(x) = f(ax)$,
$a>0$. %
 The symbol
$\lesssim $ in $f \lesssim g$ means that $f(x) \leq C g(x)$ for all
$x$ with a constant $C$ independent of $x$, and $f \asymp g$ combines
both $f \lesssim g$ and $g\lesssim f$.   Finally, $\chi_\Omega$ denotes the characteristic function of $\Omega\subset\R^d$.

\subsection{Norm differences and spectral extrema}
We will need the following elementary estimate bounding the difference
of the extreme spectral  values of two operators by the operator norm of their  difference.
For a proof of this lemma, see, e.g., \cite[Lemma~2.1]{grorosp23}.
\begin{lemma}\label{lem:norm-bd-edges}
	Let $\cH$ be  a Hilbert space, and $A_1 ,A_2 \in B(\cH)$ be
        self-adjoint operators. %
 Then
	\begin{equation}\label{eq:diff-of-norms0}
		\big|\sigma_\pm(A_1)-\sigma_\pm(A_2)\big|\leq \|A_1-A_2\|_{ B(\cH)}.
	\end{equation}
\end{lemma}

\subsection{Time-frequency representations}
We now collect some basic results on  \tfa, modulation spaces, and pseudo-differential operators that will be needed to prove our main results.
Thorough introductions to the topic   can be found in the textbook~\cite{groe1} and the two
recent monographs~\cite{BO20,CRbook}.  

For a point $z=(x,\omega)\in \R^{2d}$,  the \emph{\tfs} of $f$ is defined as
$$
\pi (z)f(t)=M_\omega T_xf(t)= e^{2\pi i \omega\cdot t}f(t-x),
$$
where $T_xf(t)=f(t-x)$ and $M_\omega f(t)=e^{2\pi i \omega\cdot t}f(t)$.
The \emph{short-time Fourier transform} of a function or distribution
$f$ on $\R^d$  with respect to a window function  $g$ is given by
\begin{align*}
V_gf(x,\omega)&=\int_{\R^d}f(t)\overline{g(t-x)}e^{-2\pi i \omega\cdot t}\ud t =\langle f,M_\omega T_xg\rangle=    \langle f,\pi (z)g\rangle .
\end{align*}
When $g$ is normalized by $\norm{2}{g}=1$, then $V_g:L^2(\mathbb{R}^d) \to L^2(\mathbb{R}^{2d})$ is an isometry %
\begin{align}\label{eq_iso}
\int_{\rdd } |\langle f, \pi(z)g\rangle |^2\, \ud z = \int_{\rdd } |V_g f(z)|^2 \, \ud z = \|f\|_2^2, \qquad f \in L^2(\mathbb{R}^d).
\end{align}
The isometry property of the \stft\, yields the following continuous resolution of the identity
\begin{equation}
\label{eq:c4}
f=\int _{\rdd} \langle f, \pi(z)g \rangle  \pi(z) g\, \ud z,\quad f\in L^2(\R^d),
\end{equation}
where the integral should be interpreted in the weak sense. Elementary calculations show  that for $z_1,z_2\in\R^{2d}$
\begin{equation}\label{eq:difference}
|\langle \pi(z_1)f,\pi(z_2)g\rangle |=|V_g(\pi(z_1)f)(z_2)|=|V_g f(z_2-z_1)|.
\end{equation}
The \emph{(cross-)Wigner distribution} of $f,g\in L^2(\R^d)$ is
\begin{align}\label{eq_wigner}
\mathcal{W}(f,g)(x,\omega)=\int_{\R^d}f\Big(x+\frac{t}{2}\Big)\overline{g\Big(x-\frac{t}{2}\Big)}e^{-2\pi i \omega\cdot t}\, \ud t.
\end{align}
If $f=g$, we write $\mathcal{W}(f)$.

\subsection{Modulation spaces} Recall the definition of  the mixed-norm modulation spaces given in \eqref{eq:def-mod-sp}. We will use the following estimate for the convolution of functions and distributions in \modsp s, taken from \cite[Prop. 2.4]{cogroe03}.
  If $f\in M^\infty(\R^d) $ and $g\in M^1_{0,s}(\R^d)$, $s\in\R$, then $f \ast g \in M^{\infty,1}_{0,s}(\R^d)$ and
 \begin{equation}
 \label{eq:c7}
 \| f \ast g\|_{M^{\infty,1}_{0,s}} \lesssim \|f \|_{M^\infty } \,
 \|g\|_{M^{1}_{0,s}}\, .
 \end{equation}
 The actual statement in \cite{cogroe03} does not allow for negative exponents $s<0$ of the weight on the right hand side above. However, as we only need the estimate for weights in the frequency variable, one can simply follow the proof in \cite{cogroe03} step by step without  significant changes.  
We will therefore not repeat the argument here.

As we note below, the space $M^\infty(\R^d)$ contains atomic measures supported on relatively separated sets.
\begin{lemma}\label{lemma_mu}
	Let $\Lambda \subset \mathbb{R}^{d}$ be relatively
        separated. Then the discrete measure  $\mu := \sum_{\lambda
          \in \Lambda} \delta_\lambda $ belongs to $ M^{\infty}(\mathbb{R}^d)$. Moreover, if $\emph{rel}(\Lambda):=\sup_{x\in\R^{d}}\#\{\lambda\in\Lambda\cap x+[0,1]^{d}\}$ then
	$$
	\|\mu\|_{M^\infty}\lesssim \emph{rel}(\Lambda).
	$$
\end{lemma}
This can be proved by a direct calculation (which is easily carried out by taking a window function $g\in M^1(\R^d)$ supported on $[0,1]^d$).

The following approximation lemma is the main technical ingredient of
our proofs. The estimates are reminiscent of notions in approximation
theory, such as the K-functional, but they seem to be new for \modsp s. 
\begin{lemma}\label{lem:2}
Let $a,b,c\in\R$, $a\leq b\leq c$.
\begin{enumerate}
\item[(i)] Let $g\in M^{\infty,1}_{0,b}(\R^d)$. For every   $0<\varepsilon\leq 1$, there exist $h=h_{\varepsilon,g}\in M^{\infty,1}_{0,c}(\R^d)$,  such that
\begin{equation*}
\|g-h\|_{M^{\infty,1}_{0,a}}\lesssim \varepsilon^{b-a}\,
\|g\|_{M^{\infty,1}_{0,b}}  \quad \text{and}\quad \|h\|_{M^{\infty,1}_{0,c}}\lesssim \varepsilon^{-(c-b)}\, \|g\|_{M^{\infty,1}_{0,b}}.
\end{equation*}
\item[(ii)] Let $g\in M^{1}_{b}(\R^d)$. For every   $0<\varepsilon\leq 1$, there exist $h=h_{\varepsilon,g}\in M^{1}_{c}(\R^d)$,  such that
\begin{equation*}
\|g-h\|_{M^{1}_{a}}\lesssim \varepsilon^{b-a}\, \|g\|_{M^{1}_{b}}  \quad \text{and}\quad \|h\|_{M^{1}_{c}}\lesssim \varepsilon^{-2(c-b)}\, \|g\|_{M^{1}_{b}}.
\end{equation*}
\end{enumerate}
\end{lemma}
(Here, the implied constants depend on $|a|$ and $|c|$.)
\begin{proof}
Let $R\geq 0$ and define $\Omega_R=\R^d\times B_{R,d}(0)$. For $R$ to be
determined later, we define   the frequency truncation of $g$ as 
\begin{equation}\label{eq:def-approx-h}
h=\int_{\R^{2d}} \chi_{\Omega_R}(z)V_\varphi g(z)\pi(z)\varphi\ud z \, .
\end{equation}
With a small computation one sees that $h$ is obtained from $g$ by a
Fourier multiplier, but for  \modsp s estimates we will need the
reproducing formula \eqref{eq:c4} and the covariance formula
\eqref{eq:difference}. Applying these identities we find that 
\begin{align*}
\|g&-h\|_{M^{\infty,1}_{0,a}}\\ &=\int_{\R^{d}}
\sup_{x\in\R^d}\left|\int_{\R^{2d}}\big(V_\varphi
                                  g(z')\hspace{-1pt}-\hspace{-1pt}\chi_{\Omega_R}(z') V_\varphi g(z')\big)\langle \pi(z')\varphi,\pi(z)\varphi\rangle \ud z'\right|\hspace{-1pt} (1+|\omega|)^a\ud \omega
\\
&\leq \int_{\R^{d}}\sup_{x\in\R^d}\int_{\R^{2d}}\chi_{\Omega_R^c}(z')|V_\varphi g(z')|\, |V_\varphi\varphi(z-z')| \ud z' (1+|\omega|)^a\ud \omega
\\
&=\big\|\big(\chi_{\Omega_R^c}\cdot|V_\varphi g|\big)\ast |V_\varphi\varphi|\big\|_{L^{\infty,1}_{0,a}}\leq \big\|\chi_{\Omega_R^c}\cdot V_\varphi g\big\|_{L^{\infty,1}_{0,a}}\, \big\| V_\varphi\varphi\big\|_{L^1_{0,|a|} }, 
\end{align*}
where we used a weighted version of Young's convolution inequality, see, for example, \cite[Proposition 11.1.3]{groe1}.
We may further estimate the first factor as 
\begin{align*}
 \big\|\chi_{\Omega_R^c} \cdot  V_\varphi g \big\|_{L^{\infty,1}_{0,a}}
  &=  \int_{\R^d\backslash B_{R,d}(0)}\sup_{x\in\R^d}|V_\varphi g(x,\omega)|  (1+|\omega|)^a \ud\omega
  \\
 &\leq (1+R)^{-(b-a)}\int_{\R^d}\sup_{x\in\R^d}|V_\varphi g(x,\omega)| (1+|\omega|)^b  \ud\omega
 \\
 &=  (1+R)^{-(b-a)}\|g\|_{M^{\infty,1}_{0,b}}.
 \end{align*}
 To bound the $M_{0,c}^{\infty,1}$-norm of $h$, we argue similarly
 \begin{align*}
 \|h\|_{M^{\infty,1}_{0,c}}&=\int_{\R^{d}}\sup_{x\in\R^d}\left|\int_{\R^{2d}}\chi_{\Omega_R}(y)V_\varphi g(z')\langle \pi(z')\varphi,\pi(z)\varphi\rangle \ud z'\right|(1+|\omega|)^c \ud \omega
 \\
&\leq \big\|\big(\chi_{\Omega_R}\cdot|V_\varphi g|\big)\ast |V_\varphi\varphi|\big\|_{L^{\infty,1}_{0,c}}\leq \big\|\chi_{\Omega_R}\cdot V_\varphi g\big\|_{L^{\infty,1}_{0,c}}\, \big\| V_\varphi\varphi\big\|_{L^1_{0,|c|} },
 \end{align*}
and 
\begin{align*}
 \big\|\chi_{\Omega_R}\cdot V_\varphi g\big\|_{L^{\infty,1}_{0,c}}&=   \int_{B_{R,d}(0)} \sup_{x\in\R^d} |V_\varphi g(x,\omega)|  (1+|\omega|)^c \ud\omega
\\
&\leq  (1+R)^{c-b}\int_{B_{R,d}(0)} \sup_{x\in\R^d} |V_\varphi g(x,\omega)|  (1+|\omega|)^b \ud\omega
\\
&=(1+R)^{c-b}\|g\|_{M^{\infty,1}_{0,b}}.
\end{align*}
 Choosing $R=\varepsilon^{-1}-1$ then yields (i).

 To prove (ii), we  set 
$\Omega_R=B_{R,2d}(0)$ (note that this is now a ball in $\rdd$). We
now  argue  as in (i). This results in
\begin{align*}
\|g-h\|_{M^{1}_{a}}\lesssim \|\chi_{\Omega_R^c}\cdot V_\varphi g\|_{L^{1}_{a}}  , \quad\text{and}\quad  \|h\|_{M^{1}_{c}}\lesssim \|\chi_{\Omega_R}\cdot V_\varphi g\|_{L^{1}_{c}}.
\end{align*}
As before, with $R=\varepsilon^{-1}-1$  we obtain  
\begin{align*}
 \big\|\chi_{\Omega_R^c} &\cdot   V_\varphi g \big\|_{L^{1}_{a}}
   =  \int_{\R^{2d}\backslash B_{R,2d}(0)} |V_\varphi g(x,\omega)|  (1+|x|)^a(1+|\omega|)^a \ud z
  \\
    &=  \int_{\R^{2d}\backslash B_{R,2d}(0)} [(1+|x|)(1+|\omega|)]^{a-b}
    |V_\varphi g(x,\omega)|  (1+|x|)^b(1+|\omega|)^b \ud z
      \\
    &\leq  \int_{\R^{2d}\backslash B_{R,2d}(0)} (1+|z|)^{a-b}
    |V_\varphi g(x,\omega)|  (1+|x|)^b(1+|\omega|)^b \ud z
  \\
 &\leq (1+R)^{-(b-a)}\int_{\R^{2d}} |V_\varphi g(x,\omega)| (1+|x|)^b(1+|\omega|)^b  \ud z
 \\
 &=  (1+R)^{-(b-a)}\|g\|_{M^{1}_{b}}=\varepsilon^{b-a}\|g\|_{M^{1}_{b}},
 \end{align*}
 where we used that
 $1+|z|=1+|(x,0)+(0,\omega)| \leq 1 + |x| + |\omega| \leq (1+|x|)(1+|\omega|)$.
Similarly,
\begin{align*}
 \big\|\chi_{\Omega_R}\cdot V_\varphi g\big\|_{L^{1}_{c}}&=   \int_{B_{R,2d}(0)}   |V_\varphi g(x,\omega)| (1+|x|)^c (1+|\omega|)^c \ud z
\\
&\leq  (1+R)^{2(c-b)}\int_{B_{R,2d}(0)} |V_\varphi g(x,\omega)| (1+|x|)^b (1+|\omega|)^b \ud z
\\
&\leq(1+R)^{2(c-b)}\|g\|_{M^{1}_{b}}=\varepsilon^{-2(c-b)}\|g\|_{M^{1}_{b}}.
\end{align*}
\end{proof}

Let us write $X_i$ to denote the multiplication operator $X_i f(t)=t_i f(t),\ 1\leq i\leq d$. It follows from straightforward calculations that 
\begin{equation}
  \label{eq:ju1}
  \|X_i \partial _if \|_{M^{p,q}_{s,t}}\lesssim   \|f
  \|_{M^{p,q}_{s+1,t+1}} \, .
\end{equation}
The following lemma can be found in \cite[Lemma~2.3]{grorosp23}. 
\begin{lemma}\label{lem:aux-W(G)}
For $g,h\in M^1_{s+t}(\rd )$,  $s,t\geq 0$, we have $\mathcal{W}(g,h)
\in M^1_{s,t}(\rdd )$ with the norm estimate 
$$
\|\mathcal{W}(g,h)\|_{M_{s,t}^1}\lesssim \|g\|_{M_{s+t}^1}\|h\|_{M_{s+t}^1} .
$$
\end{lemma}

Finally, we recall
  some norm estimates for the dilation operator on certain  modulation spaces: 
  Let $D_af(t)=f(at),\ a>0$, 
\begin{align}\label{eq_c}
	&\|D_a f\|_{M^{\infty,1}_{0,s}} \leq C_{d,s} \max\big\{1,a^{d}\big\}  \max\big\{1,a^{s}\big\} \|f\|_{M^{\infty,1}_{0,s}}, \qquad
	a>0,\ s\in\R \, ,
	\\
	\label{eq_cx}
	&\|D_a f\|_{M^{1}_{0,s}} \leq C_{d,s}   \max\big\{1,a^{-d} \big\}\max\big\{1,a^{s}\big\} \|f\|_{M^{1}_{0,s}}, \qquad\hspace{0.1cm}
	a>0,\ s\in\R \, .
\end{align}
See \cite[Theorem~1.1]{suto07} and \cite[Theorem 3.2]{coou12} for the weighted versions. 

\subsection{Pseudodifferential operators} 
The \modsp\  $M^{\infty
  ,1}(\rdd )$ is an important symbol class in the
theory of pseudodifferential operators, first used by Sj\"ostrand~\cite{Sjo94} as a class of
non-smooth  symbols that contains the H\"ormander class
$S^{0}_{0,0}$. See \cite{Gro06,GR06} for a
detailed \tfa\  of this Sj\"ostrand's class.

We will need the following properties, quoted from \cite[Theorem 2.5]{grorosp23}.
\begin{theorem}\label{lem:bdd}
(i) If $\sigma \in M^{\infty,1}(\rdd )$, then $\sigma^w $ is a bounded  operator on
$\lrd $ and
$$
\|\sigma ^w\|_{B(L^2(\R^d))} \lesssim \|\sigma
\|_{M^{\infty ,1}} \, .
$$

(ii)   Let $|\delta|<\delta_0<1$, and  $\{G_\delta\}_{ |\delta|<\delta_0}\subset
M^{\infty,1}(\R^{2d})$  be a family of real-valued symbols. If we set $T_\delta =
(D_{\sqrt{1+\delta}}G_\delta )^w$,
then 
\begin{align}\label{eq_a}
\|T_\delta\|_{B(L^2(\R^d))}\lesssim  \max\{1,(1+\delta)^{ {d} }\}
\|G_\delta\|_{M^{\infty,1}}\leq  (1+\delta_0)^{ {d} }\sup_{|t|<\delta_0}
\| {G_t}\|_{M^{\infty,1}}. 
\end{align}
\end{theorem}

\section{Proofs of the main results}\label{sec:gabor}

\subsection{H\"older continuity of extreme  spectral values}

Our proof  of Theorem~\ref{lip}   relies on the
Lipschitz continuity in the very same theorem for the case $s=2$,
which was already  established in \cite{grorosp23}, and on  a new approximation argument that may be useful in other
problems as well.   For this purpose, we approximate the Weyl
symbol $G_\delta \in M^{\infty ,1}_{0,s}$ with $s<2$ by a suitable $H_\delta$ in $M^{\infty ,1}_{0,2}$
and compare  the operator family
$T_\delta=(D_{\sqrt{1+\delta}}G_\delta)^w$ with
$U_\delta=(D_{\sqrt{1+\delta}}H_\delta)^w$. %

\begin{lemma}\label{lem:basic-estimate}
For each $|\delta|<\delta_0<1$, let $G_\delta\in M^{\infty,1}(\R^{2d})$ and $H_\delta\in M^{\infty,1}_{0,2}(\R^{2d})$ be such that $\partial_\delta H_\delta\in M^{\infty,1}(\R^{2d})$. 
Then the following holds for $\delta_1,\delta_2 \in (-\delta_0,\delta_0)$:
\begin{align*}
|\sigma_\pm(T_{\delta_1})&-\sigma_\pm(T_{\delta_2})|\lesssim (1-\delta_0)^{-(d+1)} \sup_{|t|<\delta_0} \|G_t-H_t\|_{M^{\infty,1}}\\*
 &+|\delta_1-\delta_2| (1-\delta_0)^{-(d+1)}\sup_{|t|<\delta_0} \big(\|H_t\|_{M^{\infty,1}_{0,2}}+\|\partial_tH_t\|_{M^{\infty,1}}\big).
\end{align*}
\end{lemma}
\begin{proof}
 Using the triangle inequality, we find 
\begin{align*}
|\sigma_\pm(T_{\delta_1})&-\sigma_\pm(T_{\delta_2})|\\ &\leq
                                                         |\sigma_\pm(T_{\delta_1})-\sigma_\pm(U_{\delta_1})|+|\sigma_\pm(U_{\delta_1})-\sigma_\pm(U_{\delta_2})|+|\sigma_\pm(U_{\delta_2})-\sigma_\pm(T_{\delta_2})|
                                                         \, .
\end{align*}
Since $H_\delta\in M^{\infty,1}_{0,2}(\R^{2d})$, Theorem~\ref{lip} for
$s=2$  (already proved in~\cite{grorosp23}) implies that
$$
|\sigma_\pm(U_{\delta_1})-\sigma_\pm(U_{\delta_2})| \lesssim 
|\delta_1-\delta_2| (1-\delta_0)^{-(d+1)}\sup_{|t|<\delta_0}
\big(\|H_t\|_{M^{\infty,1}_{0,2}}+\|\partial_tH_t\|_{M^{\infty,1}}\big).
$$
For the first and third term we use  Lemma~\ref{lem:norm-bd-edges} and
Theorem~\ref{lem:bdd} and bound the spectral values first by the
operator norm and then by the \modsp\ norm to obtain 
$$
 |\sigma_\pm(T_{\delta_1})-\sigma_\pm(U_{\delta_1})| \leq 
  \sup_{|t|<\delta_0}
\|D_{\sqrt{1+\delta}}(G_t-H_t)\|_{M^{\infty,1}}  \lesssim   \sup_{|t|<\delta_0}
\|G_t-H_t\|_{M^{\infty,1}} \, ,
$$
where we have used the dilation estimate  \eqref{eq_c} in the last
inequality.

The combination of  these estimate yields the claimed error bound. 
\end{proof}

This result suggests that quantifying the trade-off between the error
of approximating $G_\delta\in M^{\infty,1}_{0,s}(\R^{2d})$, $0<s\leq
2$, by a function $H_\delta\in M^{\infty,1}_{0,2}(\R^{2d})$ with
respect to the $M^{\infty,1}$-norm and the magnitude of the
$M^{\infty,1}_{0,2}$-norm of $H_\delta$ will lead to an explicit
estimate for the H\"older exponent  %
of the  extreme spectral values. The next lemma make this intuition precise.

\begin{lemma}\label{lem:approx-of-deriv}
    Let $G_\delta\in M^{\infty,1}_{0,s}(\R^{2d})$, $0<s<2$, be a family of symbols such that $\partial_\delta G_\delta \in M^{\infty,1}_{0,s-2}(\R^{2d})$. For every   $0<\varepsilon\leq 1$, there exist $H_\delta=H_{\varepsilon,G_\delta}\in M^{\infty,1}_{0,2}(\R^{2d})$ such that
\begin{enumerate}
\item[(i)] $
\|G_\delta-H_\delta\|_{M^{\infty,1}}\lesssim \varepsilon^{s}\, \sup_{|t|<\delta_0}\|G_t\|_{M^{\infty,1}_{0,s}},$ 
\item[(ii)]  $ \|H_\delta\|_{M^{\infty,1}_{0,2}}\lesssim \varepsilon^{-(2-s)}\, \sup_{|t|<\delta_0}\|G_t\|_{M^{\infty,1}_{0,s}},$ and 
\item[(iii)] 
$
 \|\partial_\delta H_\delta\|_{M^{\infty,1}}\lesssim \varepsilon^{-(2-s)}\, \sup_{|t|<\delta_0}\|\partial_t G_t\|_{M^{\infty,1}_{0,s-2}}.
 $
\end{enumerate}
(Here, the implied constants depend only on $d$.) 
\end{lemma}
\begin{proof}
Define $H_\delta$ as in \eqref{eq:def-approx-h} with
$\Omega_R=\R^{2d}\times B_{R,2d}(0)$. Properties $(i)$ and $(ii)$
are then a  direct consequence of Lemma~\ref{lem:2} with
$a=0,b=s,c=2$. 

To prove (iii), we will  show that 
\begin{equation}\label{eq:dt-Ht}
\partial_\delta H_\delta=\int_{\R^{2d}} \chi_{\Omega_R}(z)V_\varphi (\partial_\delta G_\delta)(z)\pi(z)\varphi\ud z,
\end{equation}
i.e.,  integration and  the partial derivative   may be interchanged.

Once \eqref{eq:dt-Ht} is justified,  the approximation of
$\partial_\delta G_\delta$ is obtained as the partial derivative of $H_\delta
$ and  then  the third estimate (iii) follows from
(the proof of) Lemma~\ref{lem:2}.

To  do so, we find a majorant via the following inequalities: 
\begin{align*}
  |\chi_{\Omega_R}(z)\partial_\delta  V_\varphi  G_\delta(z)&\pi(z)\varphi(t)| = |\chi_{\Omega_R}(z)V_\varphi (\partial_\delta G_\delta) (z)\pi(z)\varphi(t)|
  \\
  &\leq  (1+R)^{2-s}\sup_{|t|<\delta_0}\| V_\varphi (\partial_t G_t) \|_{L^{\infty}_{0,s-2}}\cdot |\chi_{\Omega_R}(z)\varphi(t-x)|
    \\
    &=  (1+R)^{2-s}\sup_{|t|<\delta_0}\|  \partial_t G_t \|_{M^{\infty}_{0,s-2}}\cdot |\chi_{\Omega_R}(z)\varphi(t-x)|
    \\ &\lesssim (1+R)^{2-s}\sup_{|t|<\delta_0}\|  \partial_t G_t \|_{M^{\infty,1}_{0,s-2}}\cdot |\chi_{\Omega_R}(z)\varphi(t-x)|,
\end{align*}
where we applied the embedding $M^{\infty ,1}_{0,s-2}  \hookrightarrow M^\infty_{0,s-2} $ \cite[Theorem~12.2.2]{groe1} in the last step.
The final expression is an integrable majorant as
$$
\int_{\Omega_R}|\varphi(t-x)|\ud x \ud\omega\lesssim R^{2d}\|\varphi\|_{L^1}.
$$
Equation \eqref{eq:dt-Ht} therefore follows from an application of Leibniz integral rule.
\end{proof}

\begin{proof}[Proof of Theorem~\ref{lip}]
Let us choose $\varepsilon=|\delta_1-\delta_2|^\gamma$ for some
$\gamma>0$ to be determined. The combination of  Lemma~\ref{lem:basic-estimate} and Lemma~\ref{lem:approx-of-deriv}  then yields 
\begin{align*}
&|\sigma_\pm(T_{\delta_1})-\sigma_\pm(T_{\delta_2})|\lesssim (1-\delta_0)^{-(d+1)} |\delta_1-\delta_2|^{\gamma s} \sup_{|t|<\delta_0} \|G_t\|_{M^{\infty,1}_{0,s}}
\\
 &\hspace{1.3cm}+|\delta_1-\delta_2|^{1+\gamma(s-2)} (1-\delta_0)^{-(d+1)}\sup_{|t|<\delta_0} \big(\|G_t\|_{M^{\infty,1}_{0,s}}+\|\partial_tG_t\|_{M^{\infty,1}_{0,s-2}}\big)
 \\
& \hspace{0.25cm} \leq\hspace{-1pt} (1\hspace{-1pt}-\hspace{-1pt}\delta_0)^{-(d+1)} \big(|\delta_1\hspace{-1pt}-\hspace{-1pt}\delta_2|^{\gamma s} \hspace{-1pt} +\hspace{-1pt}|\delta_1\hspace{-1pt}-\hspace{-1pt}\delta_2|^{1+\gamma(s-2)} \big)\sup_{|t|<\delta_0} \big(\|G_t\|_{M^{\infty,1}_{0,s}}\hspace{-1pt}+\hspace{-1pt}\|\partial_tG_t\|_{M^{\infty,1}_{0,s-2}}\big).
\end{align*}
On  choosing $\gamma=1/2$, the contribution of both terms is of the
same order $|\delta _1 - \delta _2|^{s/2}$, and this  concludes the proof.
\end{proof}

\subsection{Gabor frame bounds}

Let $g,h\in M^1(\R^d)$ and $\Lambda\subset \R^{2d}$ be a relatively
separated set.  The mixed Gabor frame operator of the associated set of
time-frequency  shifts using $g$ as analysis and $h$ as synthesis window  is given by
\[S_{g,h,\Lambda}f = \sum _{\lambda \in \Lambda } \langle f, \pi
(\lambda )g\rangle \pi (\lambda )h, \qquad f \in L^2(\mathbb{R}^d).\]
If $g=h$, we write $S_{g,\Lambda}=S_{g,g,\Lambda}$ to denote the Gabor frame operator. 
  The Weyl symbol of the rank one operator $\langle\, \cdot\, ,\pi(z)g\rangle \pi(z)h$ is just the shift $T_z\mathcal{W}(h,g)$, where $\mathcal{W}(h,g)$ is the cross-Wigner distribution  introduced in \eqref{eq_wigner}. Hence, the Weyl symbol of $S_{g,h,\Lambda }$ is
$$
\sigma_{g,h,\Lambda}=\sum_{\lambda\in\Lambda} T_\lambda \mathcal{W}(h,g).
$$
As we already noted before, the extreme spectral values of $S_{g,\Lambda}$ and $S_{g,\alpha\Lambda}$ are
equal to the optimal frame bounds  of $\mathcal{G}(g,\Lambda)$ and
$\mathcal{G}(g,\alpha\Lambda)$ respectively.  The Weyl symbol
corresponding to $S_{g,h,\alpha\Lambda}$ with the dilated set 
$\alpha \Lambda $  is
$$
\sigma _{g,h,\alpha \Lambda }=\sum_{\lambda\in\Lambda} T_{\alpha
  \lambda} \mathcal{W}(h,g) = D_{1/\alpha}\Big( \sum _{\lambda
  \in \Lambda } T_\lambda  D_{\alpha}\mathcal{W}(h,g)\Big)\, .
$$
We now set $\sqrt{1+\delta}=1/\alpha$, then the Weyl symbol   of the
deformed frame operator $S_{g,h,\alpha \Lambda }$ is  $\sigma
_{g,h,\alpha \Lambda }  = D_{\sqrt{1+\delta}} G_\delta$
with
$$
G_\delta = \sum _{\lambda
  \in \Lambda }   T_\lambda D_{1/\sqrt{1+\delta}}\mathcal{W}(h,g) \, , 
$$
and is therefore of the required form \eqref{eq:c1}.

Let $\mu = \sum _{\lambda \in \Lambda } \delta _\lambda $. Then
$\|\mu\|_{ M^\infty}\lesssim \text{rel}(\Lambda)$ by
Lemma~\ref{lemma_mu} and $G_\delta = \mu \ast D_{1/\sqrt{1+\delta }}\W
  (h,g)$. 
Our proof of Theorem~\ref{thm:holder} consists of two steps. First, for $1\leq s\leq 2$
we apply Theorem~\ref{lip} which requires that we bound the $M^{\infty,1}_{0,s}$-norm of the symbol and the $M^{\infty,1}_{0,s-2}$-norm of its derivative with respect to the parameter. This is done in the lemma below. For $0<s<1$ we apply an approximation argument similar to the proof of Theorem~\ref{lip}.
\begin{lemma} \label{sjosymbol}
Let $\Lambda\subset \R^{2d}$ be relatively separated, $1\leq s\leq 2$, and $0<\delta_0<1$. If $g=h\in M^1_s(\rd )$, then
\begin{enumerate}
\item[(i)] $
\|  G_\delta\|_{M^{\infty,1}_{0,s}}\lesssim \emph{rel}(\Lambda)\cdot  (1-\delta_0)^{-d}\cdot\|g\|_{M^1_s}^2$,$\qquad  \delta \in (-\delta_0,\delta_0)$,   
\item[(ii)]   $\|\partial_\delta G_\delta\|_{M^{\infty,1}_{0,s-2}}\lesssim \emph{rel}(\Lambda)\cdot  (1-\delta_0)^{-(d+2)}\cdot\|g\|_{M^1_s}^2$, $\qquad  \delta \in (-\delta_0,\delta_0)$.
\end{enumerate}
\end{lemma}

  \begin{proof}
 The
convolution relation~\eqref{eq:c7}, Lemma~\ref{lem:aux-W(G)} and the dilation
property~\eqref{eq_cx}  
--- with with $a=(1+\delta)^{-1/2} \geq (1-\delta_0)^{1/2}$ and dimension $2d$ ---
yield that 
\begin{align*}
  \|G_\delta\|_{M^{\infty ,1}_{0,s}} &=\| \mu\ast 
D_{1/\sqrt{1+\delta }} \W (g)\|_{M^{\infty ,1}_{0,s}} \lesssim \| \mu\|_{M^\infty}\|
D_{1/\sqrt{1+\delta }} \W (g)\|_{M^1_{0,s}} 
\\ &\lesssim \max \big\{ ({1\hspace{-.5pt}+\hspace{-.5pt}\delta} )^{d},   ({1\hspace{-.5pt}+\hspace{-.5pt}\delta} )^{-s/2} \big\}
\| \mu\|_{M^\infty}\|g\|_{M^1_s}^2   \lesssim \hspace{-.5pt}(1\hspace{-.5pt}-\hspace{-.5pt}\delta_0)^{-d}  \, \text{rel}(\Lambda)\, \|g\|_{M^1_s}^2.
\end{align*}
It remains to determine  $\partial_\delta G_\delta$ and estimate
its norm.
First, we note that
\begin{align*}
\partial_\delta G_\delta(z)&=\partial_{\delta}\left( \sum_{\lambda\in\Lambda}\mathcal{W}(g)\left(\frac{z-\lambda}{\sqrt{1+\delta}}\right)\right)
\\
&=  \sum_{\lambda\in\Lambda}\sum_{i=1}^{2d}-\frac{z_i-\lambda_i}{2(1+\delta)^{3/2}}\partial_i\mathcal{W}(g)\left(\frac{z-\lambda}{\sqrt{1+\delta}}\right)
\\
  &  =-\frac{1}{2(1+\delta) }\ \mu \ast
  D_{1/\sqrt{1+\delta}}\left(\sum_{i=1}^{2d}X_i\partial_i
  \mathcal{W}(g)  \right)(z)\, . 
\end{align*}
Using~\eqref{eq:c7} and~\eqref{eq_cx} as above, we prove (ii):  
\begin{align*}
\|\partial_\delta G_\delta \|_{M^{\infty,1}_{0,s-2}}&\lesssim (1-\delta_0)^{-1}
\cdot\|\mu\|_{M^\infty}\sum_{i=1}^{2d}\|D_{1/\sqrt{1+\delta}} X_i \partial_i \mathcal{W}(g)\|_{M^1_{0,s-2}}
\\
&\lesssim (1-\delta_0)^{-1-d+(s-2)/2}
\cdot \text{rel}(\Lambda) \cdot \|\mathcal{W}(g)\|_{M_{1,s-1}^1}
\\
&\lesssim (1-\delta_0)^{-(d+2)} \cdot 
\text{rel}(\Lambda)
\cdot
\|g\|_{M_s^1}^2\,  ,
\end{align*}
where  in the last step we used \eqref{eq:ju1} and Lemma~\ref{lem:aux-W(G)} which we were allowed to apply since $s-1\geq 0$.
\end{proof}

\begin{remark}
    We had to restrict to $1\leq s\leq 2$ in the previous Lemma because Lemma~\ref{lem:aux-W(G)} only guarantees  $\|\mathcal{W}(g)\|_{M_{1,s-1}^1}\lesssim  \|g\|^2_{M_{s}^1}$ if $s\geq 1$. This restriction can be relaxed with a more refined analysis, which we do not discuss here.
\end{remark}
 The next proposition proves one case of Theorem~\ref{thm:holder}. 

\begin{proposition}\label{prop:1<s<2}
  Let $\Lambda\subset \R^{2d}$ be relatively separated, $1\leq s\leq 2$, and $3/4<\alpha<2$. If $g\in M^1_s(\rd )$, then  
\begin{equation}     \left|\sigma_\pm(S_{g,\Lambda})-\sigma_\pm(S_{g,\alpha\Lambda})\right|\leq C_d\cdot \emph{rel}(\Lambda)  \cdot\|g\|_{M^1_s}^2 \cdot|1-\alpha|^{s/2}.
\end{equation}
\end{proposition}
\begin{proof}
We let $\sqrt{1+\delta } = 1/\alpha$. The restriction on the range of $\alpha$ implies that $|\delta|<\delta_0 := 7/9$, while frame operator $S_{g,\alpha \Lambda }$  has the Weyl symbol
$D_{\sqrt{1+\delta }} G_\delta $. In Lemma~\ref{sjosymbol} we have bounded the
  weighted $M^{\infty ,1}$-norm of $G_\delta $ by the $M^1_s$-norm of
  the window $g$. Therefore  Theorem~\ref{lip} yields that 
$$
     \left|\sigma_\pm(S_{g,\Lambda})-\sigma_\pm(S_{g,\alpha\Lambda})\right|\leq C_{d,s}\cdot \mathrm{rel}(\Lambda)  \cdot\|g\|_{M^1_s}^2 \cdot|\delta|^{s/2}.$$ 
Since $\delta
= \alpha ^{-2} -1$, we have 
     $$
     |\delta|^{s/2}=|\alpha^{-2}-1|^{s/2}=\alpha^{-s}(1+\alpha)^{s/2}|1-\alpha|^{s/2}\leq
     C \cdot |1-\alpha|^{s/2}, 
     $$
     as claimed. 
\end{proof}

We will need the following simple estimate for the operator norm of the Gabor frame operator.

\begin{lemma}\label{lem:norm-S}
    Let $\Lambda\subset\R^{2d}$ be relatively separated, and $3/4<\alpha<2$. If $g,h\in M^1(\R^d)$, then 
    $$
    \|S_{g,h,\alpha\Lambda}\|\lesssim \emph{rel}(\Lambda)\cdot \|g\|_{M^1}\cdot\|h\|_{M^1}.
    $$
\end{lemma}
\begin{proof}
The convolution and dilation relations \eqref{eq:c7}, \eqref{eq_c}, \eqref{eq_cx},
Theorem~\ref{lem:bdd} and Lemma~\ref{lem:aux-W(G)} imply
\begin{align*}
    \|S_{g,h,\alpha\Lambda}\|&\lesssim \|D_{1/\alpha}G_\alpha\|_{M^{\infty,1}}\lesssim \|\mu\ast D_\alpha \mathcal{W}(h,g)\|_{M^{\infty,1}}\\
     &\lesssim  \|\mu\|_{M^\infty}\|D_\alpha \mathcal{W}(h,g)\|_{M^1} \lesssim \text{rel}(\Lambda)\|g\|_{M^1}\|h\|_{M^1},
\end{align*}
as claimed.
\end{proof}

To prove the remaining case of Theorem~\ref{thm:holder}, we use an
approximation argument similar to Lemma~\ref{lem:basic-estimate}.

\begin{lemma}\label{lem:approx-frame-op}
Let $g\in M^1(\R^d)$,$h\in M^1_1(\R^d)$, and $3/4<\alpha<2$. Then 
$$
|\sigma_\pm(S_{g,\Lambda})-\sigma_\pm(S_{g,\alpha\Lambda})|\lesssim \emph{rel} (\Lambda)\Big( \big(\|g\|_{M^1}+\|h\|_{M^1}\big)\|g-h\|_{M^1} +|1-\alpha|^{1/2}\|h\|_{M^1_1}^2\Big).
$$
\end{lemma}
\begin{proof}
Using the triangle inequality and Lemma~\ref{lem:norm-bd-edges} we
infer that 
\begin{align*}
    |\sigma&_\pm(S_{g,\Lambda})-\sigma_\pm(S_{g,\alpha\Lambda})|
    \\
    &\leq
      |\sigma_\pm(S_{g,\Lambda})-\sigma_\pm(S_{h,\Lambda})|+|\sigma_\pm(S_{h,\Lambda})-\sigma_\pm(S_{h,\alpha\Lambda})|
      + |\sigma_\pm(S_{h,\alpha\Lambda})-\sigma_\pm(S_{g,\alpha\Lambda})|
    \\
    &\lesssim
      \|S_{g,\Lambda}-S_{h,\Lambda}\|+|\sigma_\pm(S_{h,\Lambda})-\sigma_\pm(S_{h,\alpha\Lambda})|+
      \|S_{h,\alpha\Lambda}-S_{g,\alpha\Lambda}\| \, .
\end{align*}
Since $h\in M^1_1(\rd )$, we can apply Proposition~\ref{prop:1<s<2} to
the second term and  obtain
$$
|\sigma_\pm(S_{h,\Lambda})-\sigma_\pm(S_{h,\alpha\Lambda})| \lesssim
\text{rel}(\Lambda)|1-\alpha|^{1/2}\|h\|_{M^1_1}^2 \, .
$$
For the estimates of the first and third term, we use
Lemma~\ref{lem:norm-S} and estimate
\begin{align*}
 \|S_{g,\alpha\Lambda}-S_{h,\alpha\Lambda}\| &=  \|S_{g,(g-h),\alpha\Lambda}+S_{(g-h),h,\alpha\Lambda}\|
    \\
    &\leq\|S_{g,(g-h),\alpha\Lambda}\|+\|S_{(g-h),h,\alpha\Lambda}\|
    \\ 
    &\lesssim
      \text{rel}(\Lambda)\Big(\big(\|g\|_{M^1}+\|h\|_{M^1}\big)\|g-h\|_{M^1}\, .
\end{align*}
The combination of both estimates yields the claimed estimate.
\end{proof}
We have now collected all necessary ingredients to prove our main two results on Gabor frame bounds.
\begin{proof}[Proof of Theorem~\ref{thm:cont}]
Clearly, we may assume that $g \not= 0$.
Let $\varepsilon>0$. As $M^1_1(\R^d)$ is dense in $M^1(\R^d)$, we may
choose $h=h_\varepsilon\in M^1_1(\R^d)$ such that $0<\|h\|_{M^1} \leq
2\|g\|_{M^1}$ and 
$$
\text{rel}(\Lambda)\big(\|g\|_{M^1}+\|h\|_{M^1}\big)\|g-h\|_{M^1}\leq \varepsilon.
$$
Subsequently, choose \[\delta^{1/2} =\min\big\{\tfrac12,\big(\text{rel}(\Lambda)\|h\|_{M^1_1}^2\big)^{-1}\, \varepsilon \big\}.\] For $|1-\alpha| < \delta$ we therefore obtain 
$\alpha \in (3/4,5/4)\subset(3/4,2)$ and
\[ \text{rel}(\Lambda)\cdot\|h\|_{M^1_1}^2 \cdot|1-\alpha|^{1/2}\leq \varepsilon.\] This, together with Lemma~\ref{lem:approx-frame-op}, implies 
$$
|\sigma_\pm(S_{g,\Lambda})-\sigma_\pm(S_{g,\alpha\Lambda})|\leq C \varepsilon, \quad \text{for }|1-\alpha| < \delta,
$$
for a constant $C$ that is independent of $g,h,\alpha$, and $\Lambda $.
\end{proof}
 
\begin{proof}[Proof of Theorem~\ref{thm:holder}]
For $1\leq s\leq 2$ the result is proven in Proposition~\ref{prop:1<s<2},
while for $s > 2$ the claim follows from the case $s=2$ (as $\|g\|_{M^1_s} \geq \|g\|_{M^1_2}$).
Let therefore $0<s<1$ and choose $\varepsilon=|1-\alpha|^\eta$, for some $\eta>0$ to be specified. A combination of Lemma~\ref{lem:approx-frame-op} and Lemma~\ref{lem:2}  shows that
$$
\big|\sigma_\pm(S_{g,\Lambda})-\sigma_\pm(S_{g,\alpha\Lambda})\big|\lesssim \text{rel}(\Lambda)\,  \big(|1-\alpha|^{s\eta}+|1-\alpha|^{1/2-2(2-2s)\eta}\big)\, \|g\|_{M^1_s}^2.
$$
We choose $\eta$ such that  the two exponents are equal, i.e., such that $$s\eta=1/2-2(2-2s)\eta.$$ Solving for $\eta$ yields $\eta=\tfrac{1}{2(4-3s)}$.
Consequently,
$$
\big|\sigma_\pm(S_{g,\Lambda})-\sigma_\pm(S_{g,\alpha\Lambda})\big|\lesssim \text{rel}(\Lambda)\, |1-\alpha|^{\tfrac{s}{2(4-3s)}}\, \|g\|_{M^1_s}^2,
$$
which concludes the proof.
\end{proof}

\end{document}